\documentclass[12pt,leqno]{article}
\usepackage{amsmath,amssymb,amscd,latexsym,paralist,bm}
\usepackage[all]{xy}

\usepackage[all]{xy}

\newcommand{\C}{{\mathbb{C}}}
\newcommand{\F}{{\mathbb{F}}}
\newcommand{\G}{\mathbb{G}}
\newcommand{\Pa}{{\mathbb{P}}}
\newcommand{\Q}{{\mathbb{Q}}}
\newcommand{\oQ}{\overline{\Q}}

\newcommand{\Z}{{\mathbb{Z}}}

\newcommand{\oZ}{\overline{\Z}}

\newcommand{\id}{\mathrm{id}}
\renewcommand{\mod}{\mathrm{mod}\,}
\newcommand{\red}{\mathrm{red}}

\newcommand{\spec}{\mathrm{spec}\,}
\newcommand{\Aut}{\mathrm{Aut}}
\newcommand{\Ext}{\mathrm{Ext}}

\newcommand{\Fr}{\mathrm{Fr}}

\newcommand{\GL}{\mathrm{GL}\,}
\newcommand{\Hom}{\mathrm{Hom}}

\newcommand{\ogamma}{\overline{\gamma}}

\newcommand{\tK}{\tilde{K}}
\newcommand{\tk}{\tilde{\kappa}}

\newcommand{\Eh}{{\mathcal E}}
\newcommand{\Fh}{{\mathcal F}}
\newcommand{\Gh}{{\mathcal G}}

\newcommand{\Oh}{{\mathcal O}}

\newcommand{\Th}{{\mathcal T}}

\newcommand{\Yh}{\mathcal{Y}}
\newcommand{\Zh}{\mathcal{Z}}

\newcommand{\emm}{{\mathfrak{m}}}
\newcommand{\eo}{\mathfrak{o}}

\newcommand{\eX}{{\mathfrak X}}

\newcommand{\ox}{\overline{x}}

\newcommand{\oy}{\overline{y}}
\newcommand{\oF}{\overline{\mathbb{F}}}

\newcommand{\tC}{\tilde{C}}

\newcommand{\iso}{\stackrel{\sim}{\longrightarrow}}
\newcommand{\blambda}{\bm{\lambda}}
\newcommand{\tgamma}{\tilde{\gamma}}
\newcommand{\verk}{\mbox{\scriptsize $\,\circ\,$}}

\newtheorem{theorem}{Theorem}

\newtheorem{prop}[theorem]{Proposition}

\newtheorem{example}[theorem]{Example}

\newenvironment{rem}{\noindent {\bf Remark}}{}

\newenvironment{proof}{\noindent {\bf Proof}}{\mbox{}\hspace*{\fill} $\Box$}
\begin{document}
\title{Representations attached to vector bundles on curves over finite and $p$-adic fields, a comparison}
\author{Christopher Deninger}
\date{\ }
\maketitle
\thispagestyle{empty}
\section{The comparison}
\label{sec:1}

In \cite{DW2} and \cite{DW4} a partial analogue of the classical Narasimhan--Seshadri correspondence between vector bundles and representations of the fundamental group was developed. See also \cite{F} for a $p$-adic theory of Higgs bundles. Let $\eo$ be the ring of integers in $\C_p = \hat{\oQ}_p$ and let $k = \eo / \emm = \oF_p$ be the common residue field of $\oZ_p$ and $\eo$. Consider a smooth projective (connected) curve $X$ over $\oQ_p$ and let $E$ be a vector bundle of degree zero on $X_{\C_p} = X \otimes \C_p$. If $E$ has potentially strongly semistable reduction in the sense of \cite{DW4} Definition 2, then for any $x \in X (\C_p)$ according to \cite{DW4} Theorem 10 there is a continuous representation
\begin{equation}
  \label{eq:1}
  \rho_{E,x} : \pi_1 (X,x) \longrightarrow \GL (E_x) \; .
\end{equation}
We now describe a special case of the theory where one can define the reduction of $\rho_{E,x} \mod \emm$. Assume that we are given the following data:

i) A model $\eX$ of $X$ i.e.~a finitely presented proper flat $\oZ_p$-scheme $\eX$ with $X = \eX \otimes_{\oZ_p} \oQ_p$,\\
ii) A vector bundle $\Eh$ over $\eX_{\eo} = \eX \otimes_{\oZ_p} \eo$ extending $E$.

Such models $\eX$ and $\Eh$ always exist. Consider the special fibre $\eX_k = \eX \otimes_{\oZ_p} k = \eX_{\eo} \otimes_{\eo} k$ and set $\Eh_k = \Eh \otimes_{\eo} k$, a vector bundle on $\eX_k$. We assume that $\Eh_k$ restricted to $\eX^{\red}_k$ is strongly semistable of degree zero in the sense of section \ref{sec:2} below.

In this case we say that $\Eh$ has strongly semistable reduction of degree zero on $\eX_{\eo}$. Then \cite{DW2} provides a continuous representation
\begin{equation}
  \label{eq:2}
  \rho_{\Eh , x_{\eo}} : \pi_1 (X,x) \longrightarrow \GL (\Eh_{x_{\eo}}) \; ,
\end{equation}
which induces \eqref{eq:1}. Here $x_{\eo} \in \eX (\eo) = X (\C_p)$ is the section of $\eX$ corresponding to $x$ and $\Eh_{x_{\eo}} = \Gamma (\spec \eo , x^*_{\eo} \Eh)$ is an $\eo$-lattice in $\Eh_x$.\\
Denoting by $x_k \in \eX_k (k) = \eX^{\red}_k (k)$ the reduction of $x_{\eo}$, we have $\Eh_{x_{\eo}} \otimes_{\eo} k = \Eh_{x_k}$ the fibre over $x_k$ of the vector bundle $\Eh_k$.

The aim of this note is to describe the reduction $\mod \emm$ of $\rho_{\Eh ,x_{\eo}}$ i.e. the representation
\begin{equation}
  \label{eq:3}
  \rho_{\Eh , x_{\eo}} \otimes k : \pi_1 (X,x) \longrightarrow \GL (\Eh_{x_k})
\end{equation}
using Nori's fundamental group scheme \cite{N}.

Let us recall some of the relevant definitions. Consider a perfect field $k$ and a reduced complete and connected $k$-scheme $Z$ with a point $z \in Z (k)$. A vector bundle $H$ on $Z$ is {\it essentially finite} if there is a torsor $\lambda : P \to Z$ under a finite group scheme over $k$ such that $\lambda^* H$ is a trivial bundle. Nori has defined a profinite algebraic group scheme $\pi (Z,z)$ over $k$ classifying the essentially finite bundles $H$ on $Z$. Every such bundle corresponds to an algebraic representation
\begin{equation}
  \label{eq:4}
  \bm{\lambda}_{H,z} : \pi (Z,z) \longrightarrow \mathbf{GL}_{H_z} \; .
\end{equation}
The group scheme $\pi (Z,z)$ also classifies the pointed torsors under finite group schemes on $Z$. If $k$ is algebraically closed, it follows that the group of $k$-valued points of $\pi (Z,z)$ can be identified with Grothendieck's fundamental group $\pi_1 (Z,z)$. On $k$-valued points the representation $\bm{\lambda}_{H,z}$ therefore becomes a continuous homomorphism 
\begin{equation}
  \label{eq:5}
  \lambda_{H,z} = \bm{\lambda}_{H,z} (k) : \pi_1 (Z,z) \longrightarrow \GL (H_z) \; .
\end{equation}
We will show the following result:

\begin{theorem}
  \label{t1}
With notations as above, consider a vector bundle $\Eh$ on $\eX_{\eo}$ with strongly semistable reduction of degree zero. Then $\Eh^{\red}_k$, the bundle $\Eh_k$ restricted to $\eX^{\red}_k$ is essentially finite. For the corresponding representation:
\[
\lambda = \lambda_{\Eh^{\red}_k , x_k} : \pi_1 (\eX^{\red}_k , x_k) \longrightarrow \GL (\Eh_{x_k}) \; ,
\]
the following diagram is commutative:
\[
\xymatrix{
\pi_1 (X,x) \ar[r]^{\rho_{\Eh , x_{\eo}} \otimes k} \ar[d] & \GL (\Eh_{x_k}) \ar@{=}[dd] \\
\pi_1 (\eX , x) \ar@{=}[d] &  \\
\pi_1 (\eX^{\red}_k , x_k) \ar[r]^{\lambda} & \GL (\Eh_{x_k}) \; .
}
\]
\end{theorem}

In particular, the reduction $\mod \emm$ of $\rho_{\Eh , x_{\eo}}$ factors over the specialization map $\pi_1 (X,x) \to \pi_1 (\eX^{\red}_k , x_k)$. In general this is not true for $\rho_{\Eh , x_{\eo}}$ itself according to Example \ref{t6}.

This note originated from a question of Vikram Mehta. I am very thankful to him and also to H\'el\`ene Esnault who once drew my attention to Nori's fundamental group.

\section{$sss$-bundles on curves over finite fields}
\label{sec:2}

In this section we collect a number of definitions and results related to Nori's fundamental group \cite{N}. The case of curves over finite fields presents some special features.

Consider a reduced complete and connected scheme $Z$ over a perfect field $k$ with a rational point $z \in Z (k)$. According to \cite{N} the $\otimes$-category of essentially finite vector bundles $H$ on $Z$ with the fibre functor $H \mapsto H_z$ is a neutral Tannakian category over $k$. By Tannakian duality it is equivalent to the category of algebraic representations of an affine group scheme $\pi (Z,z)$ over $k$ which turns out to be a projective limit of finite group schemes. 

Let $f : Z \to Z'$ be a morphism of reduced complete and connected $k$-schemes. The pullback of vector bundles induces a tensor functor between the categories of essentially finite bundles on $Z'$ and $Z$  which is compatible with the fibre functors in $f (z)$ and $z$. By Tannakian functoriality we obtain a morphism $f_* : \pi (Z,z) \to \pi (Z' , f(z))$ of group schemes over $k$. If $k$ is algebraically closed the induced map on $k$-valued points 
\[
\pi_1 (Z,z) = \pi (Z,z) (k) \to \pi (Z' , f (z)) (k) = \pi_1 (Z' , f (z))
\]
is the usual map $f_*$ between the Grothendieck fundamental groups. 

We will next describe the homomorphism
\[
\lambda_{H,z} = \blambda_{H,z} (k) : \pi_1 (Z,z) = \pi (Z,z) (k) \to \GL(H_z) 
\]
in case $H$ is trivialized by a finite \'etale covering. Consider a scheme $S$ with a geometric point $s \in S (\Omega)$. We view $\pi_1 (S,s)$ as the automorphism group of the fibre functor $F_s$ which maps any finite \'etale covering $\pi : S' \to S$ to the set of points $s' \in S'(\Omega)$ with $\pi (s') = s$.

\begin{prop}
  \label{t2n}
Let $Z$ be a reduced complete  and connected scheme over the algebraically closed field $k$ with a point $z \in Z (k)$. Consider a vector bundle $H$ on $Z$ for which there exists a connected finite \'etale covering $\pi : Y \to Z$ such that $\pi^* H$ is a trivial bundle. Then $H$ is essentially finite and the map $\lambda_{H,z} : \pi_1 (Z,z) \to \GL (H_z)$ in \eqref{eq:5} has the following description. Choose a point $y \in Y (k)$ with $\pi (y) = z$. Then for every $\gamma \in \pi_1 (Z,z)$ there is a commutative diagram:
\begin{equation}
  \label{eq:6n}
  \xymatrix{
(\pi^* H)_y \ar@{=}[d] & \Gamma (Y , \pi^* H) \ar[l]_{\overset{y^*}{\sim}} \ar[r]^{\overset{(\gamma y)^*}{\sim}} & (\pi^* H)_{\gamma y} \ar@{=}[d] \\
H_z \ar[rr]^{\lambda_{H, z} (\gamma)} && H_z 
}
\end{equation}
\end{prop}

\begin{proof}
  The covering $\pi : Y \to Z$ can be dominated by a finite \'etale Galois covering $\pi' : Y' \to Z$. Let $y' \in Y' (k)$ be a point above $y$. If the diagram \eqref{eq:6n} with $\pi , Y , y$ replaced by $\pi' , Y' , y'$ is commutative, then \eqref{eq:6n} itself commutes. Hence we may assume that $\pi : Y \to Z$ is Galois with group $G$. In particular $H$ is essentially finite. Consider the surjective homomorphism $\pi_1 (Z,z) \to G$ mapping $\gamma$ to the unique $\sigma \in G$ with $\gamma y = y^{\sigma}$. The right action of $G$ on $Y$ induces a left action on $\Gamma (Y , \pi^* H)$ by pullback and it follows from the definitions that $\lambda_{H,z}$ is the composition
\[
\lambda_{H,z} : \pi_1 (Z,z) \to G \to \GL (\Gamma (Y , \pi^* H)) \xrightarrow{\overset{\mathrm{via}\,y^*}{\sim}} \GL (H_x) \; .
\]
Now the equations
\[
(\gamma y)^* \verk (y^*)^{-1} = (y^{\sigma})^* \verk (y^*)^{-1} = (\sigma \verk y)^* \verk (y^*)^{-1} = y^* \verk \sigma^* \verk (y^*)^{-1} = \lambda_H (k) (\gamma)
\]
imply the assertion. Here $\sigma^*$ is the automorphism of $\Gamma (Y , \pi^* H)$ induced by $\sigma$. 
\end{proof}

The following class of vector bundles contains the essentially finite ones. A vector bundle $H$ on a reduced connected and complete $k$-scheme $Z$ is called strongly semistable of degree zero $(sss)$ if for all $k$-morphisms $f : C \to Z$ from a smooth connected projective curve $C$ over $k$ the pullback bundle $f^* (H)$ is semistable of degree zero, c.f. \cite{DM} (2.34). It follows from \cite{N} Lemma (3.6) that the $sss$-bundles form an abelian category. Moreover a result of Gieseker shows that it is a tensor category, c.f. \cite{Gi}. If $Z$ is purely one-dimensional, a bundle $H$ is $sss$ if and only if the pullback of $H$ to the normalization $\tC_i$ of each irreducible component $C_i$ of $Z$ is strongly semistable of degree zero in the usual sense on the smooth projective curve $\tC_i$ over $k$, see e.g. \cite{DW3} Proposition 4. 

Generalizing results of Lange--Stuhler and Subramanian slightly we have the following fact, where $\F_q$ denotes the field with $q = p^r$ elements. 

\begin{theorem}
  \label{t2}
Let $Z$ be a reduced complete and connected purely one-dimensional scheme over $\F_q$. Then the following three conditions are equivalent for a vector bundle $H$ on $Z$. \\
{\bf 1} $H$ is strongly semistable of degree zero.\\
{\bf 2} There is a finite surjective morphism $\varphi : Y \to Z$ with $Y$ a complete and purely one-dimensional scheme over $\F_q$ such that $\varphi^* H$ is a trivial bundle.\\
{\bf 3} There are a finite \'etale covering $\pi : Y \to Z$ and some $s \ge 0$ such that for the composition $\varphi : Y \xrightarrow{F^s} Y \xrightarrow{\pi} Z$ the pullback $\varphi^* H$ is a trivial bundle. Here $F = \Fr_q = \Fr^r_p$ is the $q$-linear Frobenius morphism on $Y$.

If $Z$ has an $\F_q$-rational point, these conditions are equivalent to

{\bf 4} $H$ is essentially finite.

\end{theorem}

\begin{rem}
  If $Z (\F_q) \neq \emptyset$, then according to {\bf 4} the trivializing morphism $\varphi : Y \to Z$ in {\bf 2} can be chosen to be a $G$-torsor under a finite group scheme $G / \F_q$.
\end{rem}

\begin{proof}
  The equivalence of {\bf 1} to {\bf 3} is shown in \cite{DW2} Theorem 18 by slightly generalizing a result of Lange and Stuhler. It is clear that {\bf 4} implies {\bf 2}. Over a smooth projective curve $Z / \F_q$ the equivalence of {\bf 1} and {\bf 4} was shown by Subramanian in \cite{S}, Theorem (3.2) with ideas from \cite{MS} and \cite{BPS}. His proof works also over our more general bases $Z$ and shows that {\bf 1} implies {\bf 4}. Roughly the argument goes as follows: Using the fibre functor in a point $z \in Z (\F_q)$ the abelian tensor category $\Th_Z$ of $sss$-bundles on $Z$ becomes a neutral Tannakian category over $\F_q$. Note by the way that the characterization {\bf 2} of $sss$-bundles shows without appealing to \cite{Gi} that $\Th_Z$ is stable under the tensor product. Consider the Tannakian subcategory generated by $H$. Its Tannakian dual is called the monodromy group scheme $M_H$ in \cite{BPS}. Let $n$ be the rank of $H$. The $\GL_n$-torsor associated to $H$ allows a reduction of structure group to $M_H$. Hence we obtain an $M_H$-torsor $\alpha : P \to Z$ such that $\alpha^* H$ is a trivial bundle. We have $\Fr^{s*}_q H = \Fr^{t*}_q H$ for some $s > t \ge 0$ because there are only finitely many isomorphism classes of semistable vector bundles of degree zero on a smooth projective curve over a finite field. See \cite{DW2} Proof of Theorem 18 for more details. A short argument as in \cite{S} now implies that $M_H$ is a finite group scheme and we are done.
\end{proof}

Later on we will need the following fact:

\begin{prop}
  \label{t3}
Let $S_0$ be a scheme over $\F_q$ and let $F = \Fr_q$ be the $q$-linear Frobenius morphism on $S_0$. Set $k = \overline{\F}_q$ and let $\overline{F}= F \otimes_{\F_q}$ k be the base extension of $F$ to a morphism of $S = S_0 \otimes_{\F_q} k$. Then for any geometric point $s \in S (\Omega)$ the induced map $\overline{F}_* : \pi_1 (S , s) \to \pi_1 (S , \overline{F} (s))$ is an isomorphism.
\end{prop}

\begin{proof}
  Let $F_k$ be the automorphism of $k$ with $F_k (x) = x^q$ for all $x \in k$. Then $\psi = \id_{S_0} \otimes F_k$ is an automorphism of the scheme $S$ and hence it induces isomorphisms on fundamental groups. It suffices therefore to show that
\[
(\psi \verk \overline{F})_* : \pi_1 (S , s) \to \pi_1 (S , \psi (\overline{F} (s)))
\]
is an isomorphism. The morphism $\psi \verk \overline{F}$ is the $q$-linear Frobenius morphism $\Fr_q$ on $S$. For any finite \'etale covering $\pi : T \to S$ the relative Frobenius morphism is known to be an isomorphism and hence the commutative diagram
\[
\xymatrix{
T \ar[r]^{\Fr_q} \ar[d]_{\pi} & T \ar[d]^{\pi} \\
S \ar[r]^{\Fr_q} & S
}
\]
is {\it cartesian}. It follows that $\Fr_{q*} = (\psi \verk \overline{F})_*$ is an isomorphism on fundamental groups.
\end{proof}

\section{Proof of  theorem \ref{t1}}
\label{sec:3}

For the proof of theorem \ref{t1} we first give a description of the representation $\rho_{\Eh , x_{\eo}} \otimes k$ which follows immediately from the construction of $\rho_{\Eh , x_{\eo}}$ in \cite{DW2} section 3. 

We assume that we are in the situation of theorem \ref{t1}. By assumption $\Eh^{\red}_k$ is strongly semistable of degree zero on $\eX^{\red}_k$. According to \cite{DW2} theorem 17 there is a proper morphism $\pi : \Zh \to \eX$ with the following properties:\\
{\bf a} The generic fibre $Z = \Zh \otimes_{\oZ_p} \oQ_p$ is a smooth projective connected $\oQ_p$-curve.\\
{\bf b} The induced morphism $\pi : Z \to X$ is finite and for an open dense subscheme $U \subset X$ the restriction $\pi : \pi^{-1} (U) = W \to U$ is \'etale. Moreover we have $x \in U (\C_p)$ for the chosen base point $x \in X (\C_p)$.\\
{\bf c} The scheme $\Zh$ is a model of $Z$ over $\oZ_p$ whose special fibre $\Zh_k$ is reduced. In particular $\Zh / \oZ_p$ is cohomologically flat in degree zero.\\
{\bf d} The pullback $\pi^*_k \Eh_k$ is a trivial vector bundle on $\Zh_k$.

The following construction gives a representation of $\pi_1 (U,x)$ on $\Eh_{x_k}$. For $\gamma \in \pi_1 (U,x) = \Aut (F_x)$ choose a point $z \in W (\C_p)$ with $\pi (z) = x$. Then $\gamma z \in W (\C_p)$ is another point over $x$. From $z$ and $\gamma z$ in $W (\C_p) \subset Z (\C_p)$ we obtain points $z_k$ and $(\gamma z)_k$ in $\Zh_k (k)$ as in the introduction. Consider the diagram
\begin{equation} \label{eq:6}
\Eh_{x_k} = (\pi^*_k \Eh_k)_{z_k} \xleftarrow{\overset{z^*_k}{\sim}} \Gamma (\Zh_k , \pi^*_k \Eh_k) \xrightarrow{\overset{(\gamma z)^*_k}{\sim}} (\pi^*_k \Eh_k)_{(\gamma z)_k} = \Eh_{x_k} \; .
\end{equation}
Here the pullback morphisms along $z_k : \spec k \to \Zh_k$ and $(\gamma z)_k : \spec k \to \Zh_k$ are isomorphisms because $\pi^*_k \Eh_k$ is a trivial bundle and $\Zh / \oZ_p$ is cohomologically flat in degree zero.

It turns out that the map
\begin{equation} \label{eq:7}
\rho : \pi_1 (U,x) \to \GL (\Eh_{x_k}) \quad \mbox{defined by} \; \rho (\gamma) = (\gamma z)^*_k \verk (z^*_k)^{-1}
\end{equation}
is a homomorphism of groups which (by construction) factors over a finite quotient of $\pi_1 (U,x)$. Thus $\rho$ is continuous if $\GL (\Eh_{x_k})$ is given the discrete topology. Moreover $\rho$ does not depend on either the choice of the point $z$ above $x$ nor on the choice of morphism $\pi : \Zh \to \eX$ satisfying {\bf a}--{\bf d}. It follows from \cite{DW2} Theorem 17 and Proposition 35 that $\rho$ factors over $\pi_1 (X,x)$. The resulting representation $\rho : \pi_1 (X,x) \to \GL (\Eh_{x_k})$ agrees with $\rho_{\Eh , x_{\eo}} \otimes k$. 

In order to prove theorem \ref{t1} we will now construct given $\Eh_k$ a suitable morphism $\Zh \to \eX$. We use a modification of the method from the proof of theorem 17 in \cite{DW2}. In that proof the singularities were resolved at the level of $\Yh$ which is too late for our present purposes because it creates an extension of $\Yh_k$ which is hard to control discussing the Nori fundamental group. Instead, we will resolve the singularities of a model of $X$. Then $\Yh$ does not have to be changed later. We proceed with the details:

Choose a finite extension $K / \Q_p$ with ring of integers $\eo_K$ and residue field $\kappa$ such that $(\eX , \Eh_k , x_k)$ descends to $(\eX_{\eo_K} , \Eh_0 , x_0)$. Here $\eX_{\eo_K}$ is a proper and flat $\eo_K$-scheme with $\eX_{\eo_K} \otimes_{\eo_K} \oZ_p = \eX$ and $\Eh_0$ a vector bundle on $\eX_0 = \eX_{\eo_K} \otimes \kappa$ with $\Eh_0 \otimes_{\kappa} k = \Eh_k$. Since $\Eh^{\red}_k$ is an $sss$-bundle on $\eX^{\red}_k$ the restriction $\Eh^{\red}_0$ of $\Eh_0$ to $\eX^{\red}_0$ is an $sss$-bundle as well. Finally $x_0 \in \eX_0 (\kappa)$ is a point which induces $x_k$ after base change to $k$. Theorem \ref{t2} implies that $\Eh^{\red}_0$ is essentially finite and hence $\Eh^{\red}_k$ is essentially finite as well.

After replacing $K$ by a finite extension and performing a base extension to the new $K$ we can find a semistable model $\eX'_{\eo_K}$ of the smooth projective curve $X_K = \eX_{\eo_K} \otimes K$ together with a morphism $\alpha_{\eo_K} : \eX'_{\eo_K} \to \eX_{\eo_K}$ extending the identity on the generic fibre $X_K$. This is possible by the semistable reduction theorem, c.f. \cite{A} for a comprehensive account. By Lipman's desingularization theorem we may assume that $\eX'_{\eo_K}$ besides being semistable is also regular, c.f. \cite{Lip} 10.3.25 and 10.3.26. The irreducible regular surface $\eX'_{\eo_K}$ is proper and flat over $\eo_K$.

Let $\Eh'_0$ be the pullback of $\Eh_0$ along the morphism $\alpha_0 : \eX'_0 = \eX' \otimes \kappa \to \eX_0$. Since $\eX'_0$ is reduced the map factors as $\alpha_0 : \eX'_0 \to \eX^{\red}_0 \subset \eX_0$ and $\Eh'_0$ is also the pullback of the $sss$-bundle $\Eh^{\red}_0$. Hence $\Eh'_0$ is an $sss$-bundle as well.

Using theorem \ref{t2} we find a finite \'etale covering $\pi_0 : \Yh_0 \to \eX'_0$ by a complete and one-dimensional $\kappa$-scheme $\Yh_0$ and an integer $s \ge 0$ such that under the composed map $\varphi : \Yh_0 \xrightarrow{F^s} \Yh_0 \xrightarrow{\pi_0} \eX'_0$ the pullback $\varphi^* \Eh'_0$ is a trivial bundle. Here $F = \Fr_q$ is the $q = |\kappa|$-linear Frobenius morphism on $\Yh_0$. Let $\tk$ be a finite extension of $\kappa$ such that all connected components of $\Yh_0 \otimes_{\kappa} \tk$ are geometrically connected. Let $\tK / K$ be the unramified extension with residue field $\tk$. We replace $\eX_{\eo_K}, \eX'_{\eo_K}$ and $\Eh_0 , \Eh'_0$ by their base extensions with $\eo_{\tK}$ resp. $\tk$ and $F$ by the $|\tk|$-linear Frobenius morphism. We also replace $\Yh_0$ be a connected component of $\Yh_0 \otimes_{\kappa} \tilde{\kappa}$ and $\pi_0$ by the induced morphism. Then the new $\eX_{\eo_K} , \eX'_{\eo_K} , \varphi , \ldots$ keep the previous properties and $\Yh_0$ is now geometrically connected. Using \cite{SGA1} IX Th\'eor\`eme 1.10 we may lift $\pi_0 : \Yh_0 \to \eX'_0$ to a finite \'etale morphism $\pi_{\eo_K} : \Yh_{\eo_K} \to \eX'_{\eo_K}$. The proper flat $\eo_K$-scheme $\Yh_{\eo_K}$ is regular with geometrically reduced fibres over $\eo_K$ because $\eX'_{\eo_K}$ has these properties. In particular, the morphism $\Yh_{\eo_K} \to \spec \eo_K$ is cohomologically flat in degree zero. Since the special fibre $\Yh_0$ is geometrically connected and reduced it follows that the generic fibre $Y_K$ of $\Yh_{\eo_K}$ is geometrically connected and hence a smooth projective geometrically irreducible curve over $K$. In particular $\Yh_{\eo_K}$ is irreducible in addition to being regular and proper flat over $\eo_K$. By a theorem of Lichtenbaum \cite{Li} there is thus a closed immersion $\tau_K : \Yh_{\eo_K} \hookrightarrow \Pa^N_{\eo_K}$ for some $N$. Composing with a suitable automorphism of $\Pa^N_{\eo_K}$ we may assume that $\tau^{-1}_K (\G^N_{m,K}) \subset Y_K$ contains all points in $Y_K (\C_p)$ over $x \in X_K (\C_p) = X (\C_p)$. In particular, $\tau^{-1}_K (\G^N_{ m,K})$ is open and dense in $Y_K$ with a finite complement. Thus there is an open subscheme $U_K \subset X_K$ with $x \in U_K (\C_p)$ and such that $V_K = \pi^{-1}_K (U_K)$ is contained in $\tau^{-1}_K (\G^N_{m,K})$. 

Consider the finite morphism $F_{\eo_K} : \Pa^N_{\eo_K} \to \Pa^N_{\eo_K}$ given on $A$-valued points where $A$ is any $\eo_K$-algebra, by sending $[x_0 : \ldots : x_N]$ to $[x^q_0 : \ldots : x^q_N]$. The reduction of $F_{\eo_K}$ is the $q$-linear Frobenius morphism on $\Pa^N_{\kappa}$. 

Let $\rho_{\eo_K} : \Yh'_{\eo_K} \to \Yh_{\eo_K}$ be the base change of $F^s_{\eo_K}$ via $\tau_K$. It is finite and its generic fibre $\rho_K : Y'_K \to Y_K$ is \'etale over $V_K$. Now we look at the reductions and we define a morphism $i : \Yh_0 \to \Yh'_0$ over $\kappa$ by the commutative diagram
\[
\xymatrix{
\Yh_0 \ar[dr]^i \ar@/_/[ddr]_{\tau_0} \ar@/^/[drr]^{F^s} & & \\
 & \Yh'_0 \ar[r]^{\rho_0} \ar@{^{(}->}[d] & \Yh_0 \ar@{^{(}->}[d]^{\tau_0} \\
 & \Pa^N_{\kappa} \ar[r]^{F^s} & \Pa^N_{\kappa}
}
\]
In \cite{DW2} Lemma 19 it is shown that $i$ induces an isomorphism $i : \Yh_0 \xrightarrow{\sim} \Yh^{'\red}_0$. Here the index $0$ always refers to the special fibre over $\spec \kappa$.

Taking the normalization of $\Yh'_{\eo_K}$ in the function field of an irreducible component of $Y'_K$ we get a proper, flat $\eo_K$-scheme $\Yh''_{\eo_K}$ which is finite over $\Yh'_{\eo_K}$. Its generic fibre $Y''_K$ is a smooth projective connected curve over $K$ (maybe not geometrically connected). The following diagram summarizes the situation
\[
\xymatrix{
 & & \Yh''_0 \ar[r] \ar[d] & \Yh''_{\eo_K} \ar[d] & Y''_K \ar[l] \ar[d] \\
\Yh_0 \ar[r]^{\overset{i}{\sim}} \ar[d]_{F^s} & \Yh^{'\red}_0 \ar[r] & \Yh'_0 \ar[r] & \Yh'_{\eo_K} \ar[d]_{\rho_{\eo_K}} & Y'_K \ar[l] \ar[d] \\
\Yh_0 \ar[rrr] \ar[d]_{\pi_0} & & & \Yh_{\eo_K} \ar[d]_{\pi_{\eo_K}} & Y_K \ar[l] \ar[d] \ar@{}[r]|\supset  & V_K \ar[d] \\
\eX'_0 \ar[rrr] \ar[d]_{\alpha_0} & & & \eX'_{\eo_K} \ar[d]_{\alpha_{\eo_K}} & X_K \ar[l] \ar@{=}[d]  \ar@{}[r]|\supset & U_K \\
\eX^{\red}_0 \ar[rr] & & \eX_0 \ar[r] & \eX_{\eo_K} & X_K \ar[l] 
}
\]
For a suitable finite extension $\tK / K$ all connected components of $Y''_K \otimes_K \tK$ will be geometrically connected. Let $Y'''_{\tK}$ be one of them and let $\Yh'''_{\eo_{\tK}}$ be its closure with the reduced scheme structure in $\Yh''_{\eo_K} \otimes \eo_{\tK}$. By the semistable reduction theorem there are a finite extension $L / \tK$ and a semistable model $\Zh_{\eo_L}$ of $Y'''_{\tK} \otimes_{\tK} L$ over $\Yh'''_{\eo_{\tK}}$. Base extending $\eX_{\eo_K} , \ldots , \Yh''_{\eo_K}$ over $\eo_K$ to $\oZ_p$ and $\Yh'''_{\eo_{\tilde{K}}}$ over $\eo_{\tilde{K}}$ and $\Zh_{\eo_L}$ over $\eo_L$ we get a commutative diagram, where $\delta$ is the composition $\delta : \Zh \to \Yh''' \to \Yh'' \to \Yh' \xrightarrow{\rho} \Yh$,
\begin{equation} \label{eq:8}
\xymatrix{
\Zh_k \ar[rr] \ar[d]_{\beta_k} & & \Zh \ar[dd]_{\delta} & Z \ar[l] \ar[dd]^{\delta_{\oQ_p}} \\
\Yh_k \ar[d]_{F^s \otimes_{\kappa} k} & & & \\
\Yh_k \ar[rr] \ar[d]_{\pi_k}  & & \Yh \ar[d]_{\pi} & Y \ar[l] \ar[d]^{\pi_{\oQ_p}} \ar@{}[r]|\supset & V \ar[d] \\
\eX'_k \ar[rr] \ar[d]_{\alpha_k} & & \eX' \ar[d]_{\alpha} & X \ar[l] \ar@{=}[d]  \ar@{}[r]|\supset & U \\
\eX^{\red}_k \ar@{^{(}->}[r] & \eX_k \ar[r] & \eX & X \ar[l]
}
\end{equation}
Here the morphism $\beta_k : \Zh_k \to \Yh_k$ comes about as follows: Since $\Zh_k$ is reduced, the composition $\Zh_k \to \Yh'''_k \to \Yh''_k \to \Yh'_k$ factors over $\Yh^{'\red}_k \xleftarrow{\sim} \Yh_k$ and this defines $\beta_k$. By construction, the map $\pi_{\oQ_p} \verk \delta_{\oQ_p} : Z \to X$ is finite and such that its restriction to a map $W = (\pi_{\oQ_p} \verk \delta_{\oQ_p})^{-1} (U) \to U$ is finite and \'etale. By construction the bundle $\Eh'_k = \alpha^*_k \Eh_k = \Eh'_0 \otimes_{\kappa} k$ is trivialized by pullback along $\pi_k \verk (F^s \otimes_{\kappa} k)$ and hence also along $(\pi \verk \delta)_k = \pi_k \verk (F^s \otimes_{\kappa} k) \verk \beta_k$. For later purposes note that we have a commutative diagram
\begin{equation}
  \label{eq:9}
  \xymatrix{
\Yh_k \ar[r]^{F^s \otimes_{\kappa} k} \ar[d]_{\pi_k} & \Yh_k \ar[d]^{\pi_k} \\
\eX'_k \ar[r]^{F^s \otimes_{\kappa} k} & \eX'_k
}
\end{equation}
obtained by base changing the corresponding diagram over $\kappa$:
\[
\xymatrix{
  \Yh_0 \ar[d]_{\pi_0} \ar[r]^{F^s} & \Yh_0 \ar[d]^{\pi_0} \\
\eX'_0 \ar[r]^{F^s} & \eX'_0 \; .
}
\]
The inclusion $\eX_k \to \eX$ induces a natural isomorphism $\pi_1 (\eX_k , x_k) \iso \pi_1 (\eX , x_k)$. This follows from \cite{SGA1} Exp.~X, Th\'eor\`eme~2.1 together with an argument to reduce the finitely presented case to a Noetherian one as in the proof of \cite{SGA1}, Exp. IX, Th\'eor\`eme 6.1, p. 254 above.

Next we note that there is a canonical isomorphism
\[
\pi_1 (\eX , x_k) = \Aut (F_{x_k}) = \Aut F_x = \pi_1 (\eX , x) \; .
\]
Namely, for a finite \'etale covering $\Yh \to \eX$, by the infinitesimal lifting property, any point $y_k \in \Yh_k (k)$ over $x_k$ determines a unique section $y_{\eo} \in \Yh (\eo)$ over $x_{\eo} \in \eX (\eo)$ and hence a point $y \in Y (\C_p)$ over $x \in X (\C_p)$. In this way one obtains a bijection between the points $y_k$ over $x_k$ and the points $y$ over $x$. Thus the fibre functors $F_{x_k}$ and $F_x$ are canonically isomorphic. 

Finally, by \cite{SGA1}, Exp. IX, Proposition 1.7, the inclusion $\eX^{\red}_k \hookrightarrow \eX_k$ induces an isomorphism $\pi_1 (\eX^{\red}_k , x_k) \iso \pi_1 (\eX_k , x_k)$. Thus we get an isomorphism
\[
\pi_1 (\eX^{\red}_k , x_k) \iso \pi_1 (\eX_k , x_k) = \pi_1 (\eX , x_k) = \pi_1 (\eX , x)
\]
and hence a commutative diagram
\begin{equation}
  \label{eq:10}
  \xymatrix{
 & \pi_1 (\eX' , x) \ar[dd]^{\alpha_*} \ar@{=}[r] & \pi_1 (\eX'_k , x'_k) \ar[dd]^{\alpha_{k*}} \\
\pi_1 (X,x) \ar[ur] \ar[dr] & & \\
& \pi_1 (\eX , x) \ar@{=}[r] & \pi_1 (\eX^{\red}_k , x_k) \; .
}
\end{equation}
For $\ogamma \in \pi_1 (X,x)$ choose an element $\gamma \in \pi_1 (U,x)$ which maps to $\ogamma$ and let $\ogamma_k$ be the image of $\ogamma$ in $\pi_1 (\eX'_k , x'_k)$. Fix a point $z \in W (\C_p)$ which maps to $x \in U (\C_p)$ in diagram \eqref{eq:8}. As explained at the beginning of this section the automorphism $\rho_{\Eh , x_{\eo}} (\ogamma) \otimes k$ of $\Eh_{x_k}$ is given by the formula
\begin{equation}
  \label{eq:11}
  \rho_{\Eh , x_{\eo}} (\ogamma) \otimes k = (\gamma z)^*_k \verk (z^*_k)^{-1} \; .
\end{equation}
Here the isomorphisms $z^*_k$ and $(\gamma z)^*_k$ are the ones in the upper row of the following commutative diagram, where we have set $\Fh_k = (\pi_k \verk (F^s \otimes_{\kappa} k))^* \Eh'_k$, so that $(\alpha \verk \pi \verk \delta)^*_k \Eh_k = \beta^*_k \Fh_k$. Moreover $\oy_1 := \beta_k (z_k)$ and $\oy_2 := \beta_k ((\gamma z)_k)$ in $\Yh_k (k)$,
\begin{equation}
  \label{eq:12}
  \xymatrix{
\Eh_{x_k}  \ar@{=}[r] & (\beta^*_k \Fh_k)_{z_k} \ar@{=}[d] & \Gamma (\Zh_k , \beta^*_k \Fh_k) \ar[l]_{\overset{z^*_k}{\sim}} \ar[r]^{\overset{(\gamma_z)^*_k}{\sim}} & (\beta^*_k \Fh_k)_{(\gamma_z)_k}  \ar@{=}[r] \ar@{=}[d] & \Eh_{x_k} \\
\Eh_{x_k}  \ar@{=}[r] & (\Fh_k)_{\oy_1} & \Gamma (\Yh_k , \Fh_k) \ar[l]_{\overset{\oy^*_1}{\sim}} \ar[u]^{\beta^*_k}_{\wr} \ar[r]^{\overset{\oy^*_2}{\sim}} & (\Fh_k)_{\oy_2}  \ar@{=}[r] & \Eh_{x_k} \; .
}
\end{equation}
Note here that $\Fh_k$ is already a trivial bundle and that $\Yh_k$ and $\Zh_k$ are both reduced and connected. It follows that all maps in this diagram are isomorphisms. Using \eqref{eq:11} we therefore get the formula:
\begin{equation}
  \label{eq:13}
  \rho_{\Eh , x_{\eo}} (\ogamma) \otimes k = \oy^*_2 \verk (\oy^*_1)^{-1} \; .
\end{equation}
The point $y = \delta_{\oQ_p} (z)$ in $V (\C_p) \subset Y (\C_p)$ lies above $x$ and we have $\gamma y = \delta_{\oQ_p} (\gamma z)$. Moreover the relations
\begin{equation}
  \label{eq:14}
  (F^s \otimes_{\kappa} k) (\oy_1) = y_k \quad \mbox{and} \quad (F^s \otimes_{\kappa} k) (\oy_2) = (\gamma y)_k =\ogamma_k (y_k)
\end{equation}
hold because $\gamma y = \ogamma y$ implies that $(\gamma y)_k = (\ogamma y)_k = \ogamma_k (y_k)$. Setting $\Gh_k = (F^s \otimes_{\kappa} k)^* \Eh'_k$, a bundle on $\eX'_k$, we have $\Fh_k = \pi^*_k \Gh_k$.

Next we look at representations of Nori's fundamental group. For the point $\ox_1 = \pi_k (\oy_1)$ in $\eX'_k (k)$ we have $(F^s \otimes k) (\ox_1) = x'_k$. 

Consider the commutative diagram:
\begin{equation}
  \label{eq:16}
\xymatrix{
  \pi_1 (\eX'_k , \ox_1) \ar[r]^{\lambda_{\Gh_k , \ox_1}} \ar[d]^{\wr}_{(F^s \otimes k)_*} & \GL ((\Gh_k)_{\ox_1}) \ar@{=}[d] \\
\pi_1 (\eX'_k , x'_k) \ar[r]^{\lambda_{\Eh'_k , x'_k}} \ar[d]_{\alpha_{k*}} & \GL ((\Eh'_k)_{x'_k}) \ar@{=}[d] \\
\pi_1 (\eX^{\red}_k , x_k) \ar[r]^{\lambda_{\Eh^{\red}_k , x_k}} & \GL (\Eh_{x_k}) \; .
}
\end{equation}
It is obtained by passing to the groups of $k$-valued points in the corresponding diagram for representations of Nori's fundamental group schemes. Recall that as observed above $\Eh^{\red}_k$ is an essentially finite bundle on $\eX^{\red}_k$. The fact that $(F^s \otimes k)_*$ is an isomorphism on fundamental groups was shown in Proposition \ref{t3}. Let $\tgamma_k \in \pi_1 (\eX'_k , \ox_1)$ be the element with $(F^s \otimes k)_* (\tgamma_k) = \ogamma_k$. Using the diagrams \eqref{eq:10} and \eqref{eq:16}, theorem \ref{t1} will follow once we have shown the equation
\begin{equation}
  \label{eq:17}
  \rho_{\Eh , x_{\eo}} (\ogamma) \otimes k = \lambda_{\Gh_k , \ox_1} (\tgamma_k) \quad \mbox{in} \; \GL (\Eh_{x_k}) \; .
\end{equation}

We now use the description of $\rho_{\Eh , x_{\eo}} \otimes k$ in formula \eqref{eq:13} and the one of $\lambda_{\Gh_k , \ox_1}$ in Proposition \ref{t2n} applied to the finite \'etale covering $\pi_k : \Yh_k \to \eX'_k$ which trivializes $\Gh_k$. It follows that \eqref{eq:17} is equivalent to the following diagram being commutative where we recall that $\Fh_k = \pi^*_k \Gh_k$:
\begin{equation}
  \label{eq:18}
  \xymatrix{
    \Eh_{x_k} \ar@{=}[d]  \ar@{=}[r] & (\Fh_k)_{\oy_1} & \Gamma (\Yh_k , \Fh_k) \ar[l]_{\overset{\oy^*_1}{\sim}} \ar[r]^{\overset{\oy^*_2}{\sim}}  & (\Fh_k)_{\oy_2} \ar@{=}[r] & \Eh_{x_k} \ar@{=}[d] \\
\Eh_{x_k} \ar@{=}[r] & (\Fh_k)_{\oy_1} & \Gamma (\Yh_k , \Fh_k) \ar[l]_{\overset{\oy^*_1}{\sim}} \ar[r]^{\overset{\tgamma_k (\oy_1)^*}{\sim}}  & (\Fh_k)_{\tgamma_k (\oy_1) } \ar@{=}[r] & \Eh_{x_k} \; .
}
\end{equation}
But this is trivial since we have $ \oy_2 = \tgamma_k (\oy_1)$. Namely \eqref{eq:14} implies the equations:
\[
(F^s \otimes k) (\oy_2) = \ogamma_k (y_k) = \ogamma_k ((F^s \otimes k) (\oy_1)) = (F^s \otimes k) (\tgamma_k (\oy_1)) 
\]
and $F^s \otimes k$ is injective on $k$-valued points because $F$ is universally injective.


\begin{example} 
\label{t6}
\em The following example shows that in general the representation 
\[
\rho_{\Eh , x_{\eo}} : \pi_1 (X,x) \to \GL (\Eh_{x_{\eo}})
\]
in theorem \ref{t1} does not factor over the specialization map $\pi_1 (X,x) \to \pi_1 (\eX^{\red}_k , x_k)$. Let $\eX$ be an elliptic curve over $\overline{\Z}_p$ whose reduction $\eX_k$ is supersingular. Then we have $\eX^{\red}_k = \eX_k$ and $\pi_1 (\eX_k , 0) (p) = 0$. The exact functor $E \mapsto \rho_{E,0}$ of \cite{DW2} or \cite{DW4} induces a homomorphism
\[
\rho_* : \Ext^1_{X_{\C_p}} (\Oh, \Oh) \longrightarrow \Ext^1_{\pi_1 (X,0)} (\C_p , \C_p) = \Hom (\pi_1 (X,0) , \C_p) \; .
\]
Here the second $\Ext$-group refers to the category of finite dimensional $\C_p$-vector spaces with a continuous $\pi_1 (X,0)$-operation. Moreover, $\Hom$ refers to continuous homomorphisms. In \cite{DW1} Corollary 1, by comparing with Hodge--Tate theory it is shown that $\rho_*$ is injective. For an extension of vector bundles $0 \to \Oh \to E \to \Oh \to 0$ on $X_{\C_p}$ the corresponding representation $\rho_{E,0}$ of $\pi_1 (X,0)$ on $\GL (E_0)$ is unipotent of rank $2$ and described by the additive character
\[
\rho_* ([E]) \in \Hom (\pi_1 (X,0) , \C_p) = \Hom (\pi_1 (X,0) (p) , \C_p) \; .
\]
In particular $\rho_{E,0}$ factors over $\pi_1 (X,0) (p)$ and $\rho_{E,0}$ is trivial if and only if $[E] = 0$. Thus any extension $[\Eh]$ in $H^1 (\eX , \Oh)$ whose restriction to $H^1 (X, \Oh)$ is non-trivial has a non-trivial associated representation
\[
\rho_{\Eh , 0} : \pi_1 (X,0) \longrightarrow \GL (\Eh_0) \; .
\]
Since $\rho_{\Eh,0}$ factors over $\pi_1 (X,0) (p)$ it cannot factor over $\pi_1 (\eX_k ,0)$ because then it would factor over $\pi_1 (\eX_k ,0) (p) = 0$.
\end{example}


\begin{thebibliography}{999}
\bibitem[A]{A} {Abbes, Ahmed}, {R\'eduction semi-stable des courbes d'apr\`es {A}rtin,
              {D}eligne, {G}rothendieck, {M}umford, {S}aito, {W}inters,
              {$\ldots$}}. In: {Courbes semi-stables et groupe fondamental en g\'eom\'etrie
              alg\'ebrique ({L}uminy, 1998)}, {Progr. Math.}, {\bf 187}, {59--110}, {Birkh\"auser}, {Basel}, {2000}
\bibitem[BPS]{BPS} I. Biswas, A.J. Parameswaran, S. Subramanian, Monodromy group for a strongly semistable principal bundle over a curve. Duke Math. J. {\bf 132} (2006), 1--48
\bibitem[DM]{DM} P. Deligne, J.S. Milne, Tannakian categories. Springer LNM 900
\bibitem[DW1]{DW1} C. Deninger, A. Werner, Line bundles and $p$-adic characters. In: G. van der Geer, B. Moonen, R. Schoof (eds.), Number Fields and Function Fields -- Two Parallel Worlds, Progress in Mathematics {\bf 239}, 101--131, Birkh\"auser 2005
\bibitem[DW2]{DW2} C. Deninger, A. Werner, Vector bundles on $p$-adic curves and parallel transport. Ann. Scient. \'Ec. Norm. Sup. {\bf 38} (2005), 553--597
\bibitem[DW3]{DW3} C. Deninger, A. Werner, On Tannakian duality for vector bundles on $p$-adic curves. In: Algebraic cycles and motives, vol. 2, LMS Lecture notes {\bf 344} (2007), 94--111
\bibitem[DW4]{DW4} C. Deninger, A. Werner, Vector bundles on $p$-adic curves and parallel transport II. Preprint 2009, arXiv: 0902.1437
\bibitem[F]{F} G. Faltings, A $p$-adic Simpson correspondence. Adv. Math. {\bf 198} (2005), 847--862
\bibitem[Gi]{Gi} D. Giesecker, On a theorem of Bogomolov on Chern classes of stable bundles. Am. J. Math. {\bf 101} (1979), 79--85
\bibitem[Li]{Li} S. Lichtenbaum, Curves over discrete valuation rings. Amer. J. Math. {\bf 90} (1968), 380--405
\bibitem[Lip]{Lip} J. Lipman, Desingularization of two-dimensional schemes. Ann. of Math. {\bf 107} (1978), 151--207
\bibitem[MS]{MS} V.B. Mehta, S. Subramanian, On the fundamental group scheme. Invent. Math. {\bf 148} (2002), 143--150
\bibitem[N]{N} M.V. Nori, The fundamental group-scheme. Proc. Indian Acad. Sci. {\bf 91} (1982), 73--122
\bibitem[S]{S} S. Subramanian, Strongly semistable bundles on a curve over a finite field. Arch. Math. {\bf 89} (2007), 68--72
\bibitem[SGA1]{SGA1} A. Grothendieck et al., Rev\^etements \'etales et groupe fondamental. Springer LNM 224, 1971
\end{thebibliography}
\end{document}